\documentclass[11pt]{article}
\usepackage[english]{babel}
\usepackage{amsmath}
\usepackage{amsfonts}
\usepackage{amssymb}
\usepackage{amsthm}
\usepackage{latexsym}
\usepackage[final]{showkeys}
\newtheorem{theorem}{Theorem}
\newtheorem{proposition}{Proposition}
\newtheorem{lemma}{Lemma}

\def\neweq#1{\begin{equation}\label{#1}}
\def\endeq{\end{equation}}
\def\eq#1{(\ref{#1})}

\newcommand{\R}{\mathbb{R}}
\newcommand{\eps}{\varepsilon}
\newcommand{\weak}{\rightharpoonup}

\begin{document}

\title{Supercritical biharmonic equations
with power-type nonlinearity}

\author{Alberto Ferrero \\
Dipartimento di Matematica\\
Universit\`{a} di Milano-Bicocca\\
Via Cozzi 53\\
20125 Milano (Italy)
\and
Hans-Christoph Grunau\\
Fakult\"at f\"ur Mathematik\\
Otto-von-Guericke-Universit\"at\\
Postfach 4120\\
39016 Magdeburg (Germany)
\and
Paschalis Karageorgis\\
School of Mathematics\\
Trinity College\\
Dublin 2 (Ireland)
}

\date{}

\maketitle

\begin{abstract}
The biharmonic supercritical equation $\Delta^2u=|u|^{p-1}u$,
where $n>4$  and $p>(n+4)/(n-4)$, is studied in the whole
space $\mathbb{R}^n$ as well as in a modified form with
$\lambda(1+u)^p$ as right-hand-side with
an additional eigenvalue parameter $\lambda>0$
in the unit ball, in the latter case
together with Dirichlet boundary conditions. As for entire regular
radial solutions we prove oscillatory behaviour around the explicitly known
radial {\it singular} solution, provided $p\in((n+4)/(n-4),p_c)$, where
$p_c\in ((n+4)/(n-4),\infty]$ is a further critical exponent, which was
introduced in a recent work by Gazzola and the second author.
The third author proved already that these oscillations do not occur in
the complementing case, where $p\ge p_c$.

Concerning the Dirichlet problem we prove existence of at least one
singular solution with corresponding eigenvalue parameter. Moreover,
for the extremal solution in the bifurcation diagram for this nonlinear
biharmonic eigenvalue problem, we prove smoothness as long as $p\in((n+4)/(n-4),p_c)$.
\end{abstract}

\section{Introduction and main results}

In the present paper we consider  qualitative properties of
entire radial solutions (defined and regular in the whole space)
of the supercritical biharmonic equation
\begin{equation}\label{pde}
\Delta^2u=|u|^{p-1}u\qquad\mbox{in }\mathbb{R}^n,
\end{equation}
where $n\ge5$ and $p>\frac{n+4}{n-4}$.
An important role is played by the explicitly known entire solution
\begin{equation}\label{singularsolution}
u_s(r) =K_0^{1/(p-1)} r^{-4/(p-1)},
\end{equation}
where
\begin{equation}
K_0=\frac{4}{p-1}\left( \frac{4}{p-1}+2\right)
\left(n-2- \frac{4}{p-1}\right)\left(n-4- \frac{4}{p-1}\right).
\end{equation}
It was shown in \cite{dalmasso,GazzolaGrunau} that 
positive regular entire solutions to
(\ref{pde}) exist and that asymptotically they behave like the
singular solution $u_s$:
$$
\lim_{r\to\infty} \frac{u(r)}{u_s(r)}=1.
$$
Moreover, for $n>12$ a further critical exponent
$p_c\in\left(\frac{n+4}{n-4},\infty \right)$ was introduced being
in that interval the unique solution of the following
polynomial equation:
\begin{equation}\label{definepc}
p_c\cdot\frac{4}{p_c-1}\cdot\left( \frac{4}{p_c-1}+2\right)
\cdot\left(n-2- \frac{4}{p_c-1}\right)\cdot\left(n-4- \frac{4}{p_c-1}\right)
=\frac{n^2(n-4)^2}{16}.
\end{equation}
The third author \cite{Karageorgis} proved in particular that
in the ``supercritical case'', i.e
$$
p\ge p_c
$$
the convergence
of $u$ to $u_s$ is monotone, i.e. $\forall r:\ u(r)<u_s(r)$.
Here, we study the reverse case:

\begin{theorem}\label{theorem1}
Let $p_c\in((n+4)/(n-4),\infty)$ be the number, which is defined by
(\ref{definepc}) for $n\ge 13$. We assume that
$$
\frac{n+4}{n-4} <p<p_c \mbox{\ if\ }n\ge 13,\qquad
\frac{n+4}{n-4} <p<\infty \mbox{\ if\ }5\le n\le12.
$$
Let $r\mapsto u(r)$ be a radial entire solution to (\ref{pde}). Then,
as $r\to\infty$, $u(r)$ oscillates infinitely many times around the singular solution $u_s(r)$.
\end{theorem}

We study also existence of singular solutions as well as
qualitative properties of positive solutions
of  the corresponding Dirichlet problem
\begin{equation}\label{Dirichlet}
\left\{\begin{array}{ll}
\Delta^2u=\lambda (1+u)^ p \ & \mbox{in }B,\\
u>0 &   \mbox{in }B,\\
u=|\nabla u|=0 & \mbox{on }\partial B,
\end{array}\right.
\end{equation}
where $B\subset \mathbb{R}^n$ is the unit ball, $\lambda>0$ is an
eigenvalue parameter and again
$n\ge5$ and $p>\frac{n+4}{n-4}$.
In \cite{FerreroGrunau} (see also \cite{bg})
it was proved that there exists an extremal parameter
$\lambda^*$ such that for $\lambda\in[0,\lambda^*)$ one has a minimal
solution which is regular, while not even a weak solution does exist for $\lambda
>\lambda^*$. On the extremal parameter $\lambda=\lambda^*$, an extremal
solution $u^*\in H^2_0(B)\cap L^p(B)$ exists as monotone limit of the minimal solutions.
It is expected that also in the Dirichlet problem, a singular  (i.e. unbounded)
solution $u_{\sigma}$ corresponding
to a suitable singular parameter $\lambda_{\sigma}$ exists and will play an important role
as far as the shape of the bifurcation diagram for (\ref{Dirichlet}) is concerned.
However, in  \cite{FerreroGrunau} we had to leave open even the existence of a singular
solution which will be proved in the present paper:

\begin{theorem} \label{exss}
Let $n>4$ and $p>(n+4)/(n-4)$. Then, there exists a parameter $\lambda_{\sigma}>0$ such that
for $\lambda=\lambda_{\sigma}$, problem $(\ref{Dirichlet})$ admits a radial singular solution.
\end{theorem}

Moreover, in  \cite{FerreroGrunau} we
 left open whether the extremal solution $u^*$ introduced above
 is singular (unbounded) or regular (bounded).
The corresponding question has been settled for the exponential nonlinearity by
Davila, Dupaigne, Guerra and Montenegro \cite{DDGM} thereby developing
the previous work \cite{aggm}.
Here, taking advantage of an idea in \cite{DDGM}, we prove regularity
of the extremal solution of the problem with power-type nonlinearity in the ``subcritical'' range.

\begin{theorem}\label{theorem2}
Let $p_c\in((n+4)/(n-4),\infty)$ be the number, which is defined by
(\ref{definepc}) for $n\ge 13$. We assume that
$$
\frac{n+4}{n-4} <p<p_c \mbox{\ if\ }n\ge 13,\qquad
\frac{n+4}{n-4} <p<\infty \mbox{\ if\ }5\le n\le12.
$$
Let $ u^*\in H^2_0(B)\cap L^p(B)$ be the extremal radial solution of  (\ref{Dirichlet})
corresponding to the extremal parameter $ \lambda^*$, which is obtained
as monotone limit of the minimal regular solutions for $\lambda\nearrow\lambda^*$.
Then, $u^*$ is regular.
\end{theorem}

Related results for the corresponding second order problems were obtained e.g.
in \cite{bcmr,bv,mignot,W}.

\section{Entire solutions: The corresponding autonomous system}

Here we study qualitative properties of entire radial solutions $r\mapsto u(r)$
to (\ref{pde}) and shall prove Theorem~\ref{theorem1}.
We put
\begin{equation}
v(s):=e^{4s/(p-1)}u(e^s)\quad (s\in \mathbb{R}),\qquad
u(r)=r^{-4/(p-1)} v(\log r) \quad (r>0).
\end{equation}
According to \cite{GazzolaGrunau,Karageorgis}, (\ref{pde}) is then
equivalent to
\begin{equation}\label{ode}
\left(\partial_s  -\frac{4}{p-1}+n-4\right)\left( \partial_s  -\frac{4}{p-1}+n-2\right)
\left(\partial_s   -\frac{4}{p-1}-2\right)\left( \partial_s  -\frac{4}{p-1}\right)v(s)
=|v(s)|^{p-1}v(s),
\end{equation}
$s\in \mathbb{R}$.
In order to write this as an autonomous system, we define
\begin{equation}
\left\{\begin{array}{rcl}
w_1(s) &=& v(s)\\
w_2(s) &=&\left( \partial_s  -\frac{4}{p-1}\right)w_1(s)\\
w_3(s)&=&\left(\partial_s   -\frac{4}{p-1}-2\right)w_2(s)\\
w_4(s)&=& \left( \partial_s  -\frac{4}{p-1}+n-2\right)w_3(s).
\end{array}\right.
\end{equation}
Equation (\ref{ode}) is equivalent to the following system:
\begin{equation}\label{autonomous}
\left\{\begin{array}{rcl}
w_1'(s) &=&\frac{4}{p-1}w_1+w_2,\\
w_2'(s) &=&\left(\frac{4}{p-1}  +2\right)w_2+w_3,\\
w_3'(s)&=&\left(\frac{4}{p-1}  -(n-2)\right)w_3+w_4,\\
w_4'(s)&=&|w_1(s)|^{p-1}w_1(s)+\left(\frac{4}{p-1}-(n-4)  \right)w_4.
\end{array}\right.
\end{equation}
In order to perform the stability analysis around the singular solution
$u_s(r)=K_0^{1/(p-1)}r^{-4/(p-1)}$, i.e. $v(s)=K_0^{1/(p-1)}$,
we have to linearize (\ref{autonomous})
around the vector
$$
w^{(0)}:= K_0^{1/(p-1)}\left(1,-\frac{4}{p-1},
\frac{4}{p-1}\left(\frac{4}{p-1}+2\right),
\left( n-2-\frac{4}{p-1} \right)
\frac{4}{p-1}\left(\frac{4}{p-1}+2\right)
\right)
$$
and come up with the system $w'(s)=M\circ w(s)$ where
$$
M:=\begin{pmatrix}
\frac{4}{p-1}&1&0&0\\
0& \frac{4}{p-1}+2&1&0\\
0&0&\frac{4}{p-1}-(n-2)&1\\
pK_0 &0&0&\frac{4}{p-1}-(n-4)\\
\end{pmatrix}.
$$
The corresponding characteristic polynomial is given by
$$
P(\nu)=\left(\nu  -\frac{4}{p-1}+n-4\right)\left( \nu  -\frac{4}{p-1}+n-2\right)
\left(\nu   -\frac{4}{p-1}-2\right)\left( \nu  -\frac{4}{p-1}\right)-pK_0.
$$
According to \cite{GazzolaGrunau},  the eigenvalues are given by
\begin{eqnarray*}
\nu_1= \frac{N_1 + \sqrt{N_2 + 4\sqrt{N_3} } }{2(p-1)},&\quad&
\nu_2= \frac{N_1 -  \sqrt{N_2 +  4\sqrt{N_3} } }{2(p-1)},\\
\nu_3= \frac{N_1 + \sqrt{N_2  - 4\sqrt{N_3} } }{2(p-1)},&\quad&
\nu_4= \frac{N_1 -  \sqrt{N_2  - 4\sqrt{N_3} } }{2(p-1)},
\end{eqnarray*}
where
$$N_1:=-(n-4)(p-1)+8,\qquad N_2:=(n^2 -4n +8) (p-1)^2,$$
\begin{eqnarray*}
N_3&:=& (9n-34)(n-2)\, (p-1)^4+8(3n-8)(n-6)\,(p-1)^3\\
         &&   +(16 n^2 -288n +  832) \, (p-1)^2 - 128 (n-6)(p-1) +256.
\end{eqnarray*}
One has $\nu_1,\nu_2\in\R$ and $\nu_2<0<\nu_1$.
For any $5\le n\le12$ we have $\nu_3,\nu_4\not\in\R$ and
{\rm Re}\,$\nu_3=\,${\rm Re}\,$\nu_4<0$.
For any $n\ge13$ and  $p<p_c$,  $\nu_3,\nu_4\not\in\R$ and
{\rm Re}\,$\nu_3=\,${\rm Re}\,$\nu_4<0$, while $\nu_3,\nu_4\in\R$ and $\nu_4\le\nu_3<0$
if $p\ge p_c$. In any case,
$$
\nu_2 <\operatorname{Re} \nu_{3/4}<0<\nu_1.
$$
The stable manifold of $w^{(0)}$, where the trajectory of any $w$ corresponding
to an entire regular solution is contained in, is tangential to the span of the eigenvectors
corresponding to $\nu_{2},\nu_3,\nu_4$.
In \cite{GazzolaGrunau} the following strategy to prove Theorem~\ref{theorem1}
was outlined: in the ``subcritical'' setting
$\frac{n+4}{n-4}<p<p_c$, any such trajectory oscillates around $w^{(0)}$
infinitely many times except those which are tangential to the eigenvector
corresponding to $\nu_2$. We show that the latter can not correspond to an
entire regular solution.

\begin{proposition}\label{impossible}
Let $w(\,.\,) $ be a solution of (\ref{autonomous}) in the stable manifold
of $w^{(0)}$ being tangential to the eigenvector
corresponding to $\nu_2$. Then the corresponding solution $u$ of (\ref{pde})
is singular or even not defined for all $r>0$.
\end{proposition}

In order to prove this proposition we need the following crucial observation
on the sign of the components of an eigenvector corresponding to $\nu_2$:

\begin{lemma}
One eigenvector of $M$ corresponding to $\nu_2$ is given by $t=(t_1,t_2,t_3,t_4)$
with
\begin{eqnarray*}
t_1&=&1>0,\\
t_2&=&\left( \nu_2-\frac{4}{p-1}\right)<0,\\
t_3&=&\left( \nu_2-2-\frac{4}{p-1}\right)\left( \nu_2-\frac{4}{p-1}\right)>0,\\
t_4&=&\left( \nu_2+n-2-\frac{4}{p-1}\right)
\left( \nu_2-2-\frac{4}{p-1}\right)\left( \nu_2-\frac{4}{p-1}\right)<0.
\end{eqnarray*}
\end{lemma}

\begin{proof}
Since $\nu_2<0$ we only have to show that
\begin{equation}\label{negative}
0>\nu_2+n-2-\frac{4}{p-1}=\frac{n}{2}-\frac{1}{2(p-1)}
\sqrt{N_2+4\sqrt{N_3}}
\end{equation}
the latter being equivalent to proving that
$$
N_3>(n-2)^2(p-1)^4.
$$
Indeed, by using the supercriticality assumption $(n-4)(p-1)>8$, we have
\begin{eqnarray*}
N_3-(n-2)^2(p-1)^4 &=& 8(n-2)(n-4)(p-1)^4+8(3n^2-26n+48)(p-1)^3\\
&&+16(n^2-18n+52)(p-1)^2-128(n-6)(p-1)+256\\
&=& 8p(p+1)((n-2)(p-1)-4)((n-4)(p-1)-4)>0.
\end{eqnarray*}
This proves (\ref{negative}) and hence the lemma.
\end{proof}

\vspace{4mm}\noindent
{\it Proof of Proposition~\ref{impossible}.}
Let $w(\, .\,)$ be a solution to (\ref{autonomous}) being tangential  for $s\to\infty$
to the eigenvector $t$ from the previous lemma. We may assume that $w(\, .\,)$
exists on the whole real line $\mathbb{R}$ because otherwise, nothing is to be proved.
We put $z_1(s)=w_1(s)-w_1^{(0)}$ and further
\begin{eqnarray*}
z_1(s) &=& w_1(s) -w_1^{(0)}=v(s)-K_0^{1/(p-1)},\\
z_2(s) &=&w_2(s) -w_2^{(0)}=\left( \partial_s -\frac{4}{p-1}\right)z_1(s),\\
z_3(s)&=&w_3(s) -w_3^{(0)}=\left( \partial_s -\frac{4}{p-1}-2\right)z_2(s),\\
z_4(s)&=&w_4(s) -w_4^{(0)}=\left( \partial_s -\frac{4}{p-1}+n-2\right)z_3(s),\\
\end{eqnarray*}
so that
$$
\left( \partial_s -\frac{4}{p-1}+n-4\right)z_4(s)=|v(s)|^{p-1}v(s)-K_0^{p/(p-1)}
=|w_1(s)|^{p-1}w_1(s)-|w_1^{(0)}|^{p-1}w_1^{(0)}.
$$
Writing this more systematically yields
\begin{equation}\left\{\begin{array}{rcl}
z_1'(s) &=&\frac{4}{p-1}z_1(s)+z_2(s),\\
z_2'(s) &=&\left( \frac{4}{p-1}+2\right)z_2(s)+z_3(s),\\
z_3'(s) &=&\left( \frac{4}{p-1}-(n-2)\right)z_3(s)+z_4(s),\\
z_4'(s)&=& |w_1(s)|^{p-1}w_1(s)-|w_1^{(0)}|^{p-1}w_1^{(0)}
+\left( \frac{4}{p-1}-(n-4)\right)z_4(s).
\end{array}\right.
\end{equation}
According to whether $z(\, .\,)$ approaches the origin from ``above'' or
``below'' we distinguish two cases.

\vspace{2mm}\noindent
{\it First case.} There exists $s_0$ large enough such that
\begin{equation}\label{positive}
z_1(s_0) >0, \quad z_2(s_0)<0,\quad z_3(s_0) >0, \quad z_4(s_0)<0.
\end{equation}
On any interval $[s,s_0]$ where $z_1(\, .\,)=w_1(\, .\,)-w_1^{(0)}\ge 0$, we
must then have
\begin{equation*}
\left( \partial_s+(n-4)-\frac{4}{p-1}\right)z_4(s)=
|w_1(s)|^{p-1}w_1(s)-|w_1^{(0)}|^{p-1}w_1^{(0)}\ge 0.
\end{equation*}
This makes $e^{\left( (n-4)-\frac{4}{p-1}\right)s} z_4(s)$ increasing on $[s,s_0]$,
and so \eqref{positive} implies that
\begin{equation*}
e^{\left( (n-4)-\frac{4}{p-1}\right)s} z_4(s) \leq
e^{\left( (n-4)-\frac{4}{p-1}\right)s_0} z_4(s_0) < 0
\end{equation*}
on $[s,s_0]$.  In particular, $z_4(s) <0$ throughout the interval, and we have
\begin{equation*}
\left( \partial_s+(n-2)-\frac{4}{p-1}\right)z_3(s) = z_4(s) < 0.
\end{equation*}
This makes $e^{\left( (n-2)-\frac{4}{p-1}\right)s} z_3(s)$ decreasing on $[s,s_0]$,
so we similarly find that
\begin{equation*}
e^{\left( (n-2)-\frac{4}{p-1}\right)s} z_3(s) \geq
e^{\left( (n-2)-\frac{4}{p-1}\right)s_0} z_3(s_0) > 0
\end{equation*}
by \eqref{positive}.  Since $\left( \partial_s-2-\frac{4}{p-1}\right)z_2(s) = z_3(s) >0$,
the exact same argument leads us to
\begin{equation*}
e^{\left( -2-\frac{4}{p-1}\right)s} z_2(s) \leq
e^{\left( -2-\frac{4}{p-1}\right)s_0} z_2(s_0) < 0
\end{equation*}
by \eqref{positive}, hence $\left( \partial_s-\frac{4}{p-1}\right)z_1(s) = z_2(s) <0$
and we finally get
\begin{equation*}
e^{-\frac{4}{p-1} s} z_1(s) \geq
e^{-\frac{4}{p-1} s_0} z_1(s_0) >0.
\end{equation*}
That is, $z_1(s)>0$ on any interval $[s,s_0]$ where $z_1(s)\geq 0$, so it is impossible
for $z_1(s)$ to become $0$ at some $s<s_0$. Hence $\forall s\leq s_0:\quad z_1(s)>0$. For
the original solution this means that for $r\le r_0$, $u(\, .\, )$ lies above the singular
solution. This means that $u(\,.\,)$ itself is singular at $r=0$.

\vspace{2mm}\noindent
{\it Second case.} There exists $s_0$ large enough such that
\begin{equation}\label{Negative}
z_1(s_0) <0, \quad z_2(s_0)>0,\quad z_3(s_0) <0, \quad z_4(s_0)>0.
\end{equation}
On any interval $[s,s_0]$ where $z_1(\, .\,)=w_1(\, .\,)-w_1^{(0)}\le 0$, we
must then have
\begin{equation*}
\left( \partial_s+(n-4)-\frac{4}{p-1}\right)z_4(s)=
|w_1(s)|^{p-1}w_1(s)-|w_1^{(0)}|^{p-1}w_1^{(0)}\le 0.
\end{equation*}
This makes $e^{\left( (n-4)-\frac{4}{p-1}\right)s} z_4(s)$ decreasing on $[s,s_0]$,
and so \eqref{Negative} implies that
\begin{equation*}
e^{\left( (n-4)-\frac{4}{p-1}\right)s} z_4(s) \geq
e^{\left( (n-4)-\frac{4}{p-1}\right)s_0} z_4(s_0) > 0
\end{equation*}
on $[s,s_0]$. In particular, $z_4(s)>0$ throughout the interval, and we have
\begin{equation*}
\left( \partial_s+(n-2)-\frac{4}{p-1}\right)z_3(s) = z_4(s) > 0.
\end{equation*}
This makes $e^{\left( (n-2)-\frac{4}{p-1}\right)s} z_3(s)$ increasing on $[s,s_0]$,
so we similarly find that
\begin{equation}\label{blowup}
e^{\left( (n-2)-\frac{4}{p-1}\right)s} z_3(s) \leq
e^{\left( (n-2)-\frac{4}{p-1}\right)s_0} z_3(s_0) < 0
\end{equation}
by \eqref{Negative}.  Following this approach, as in the first case, we eventually get
\begin{equation}\label{signs}
z_4(s) > 0, \quad z_3(s)<0, \quad z_2(s)>0, \quad z_1(s)<0
\end{equation}
on any interval $[s,s_0]$ where $z_1(s)\leq 0$, so it is impossible for $z_1(s)$ to
become $0$ at some $s<s_0$. Hence $\forall s\leq s_0:\quad z_1(s)<0$, i.e. the corresponding
$u(\,.\,)$ is always below the singular solution. In order to prove that $u(\,.\,)$ itself
is singular also in this case, we show that $z_1(s)\to-\infty$ for $s\to-\infty$.
Since $\forall s\le s_0:\quad z_1(s)<0$, we have that (\ref{blowup}) holds
true for all $s\le s_0$. Referring to \cite[Proposition 1]{FerreroGrunau} would already show that
also $v$ and so $u$ cannot be bounded. However, here it is quite easy
to show this directly. For some suitable constant $\delta_1>0$ one has:
\begin{equation*}
\partial_s \left(  e^{-\left( 2+\frac{4}{p-1}\right)s}z_2(s)\right) =
e^{-\left( 2+\frac{4}{p-1}\right)s} z_3(s) \leq -\delta_1 e^{-ns}
\end{equation*}
because of \eqref{blowup}, and this implies that
\begin{align*}
e^{-\left( 2+\frac{4}{p-1}\right)s}z_2(s)
&\geq \frac{\delta_1}{n} e^{-ns} - \frac{\delta_1}{n} e^{-ns_0}
+ e^{-\left( 2+\frac{4}{p-1}\right)s_0}z_2(s_0) \\
&\geq \delta_2 e^{-ns}
\end{align*}
for some suitable constant $\delta_2>0$.  In particular,
\begin{equation*}
\partial_s \left(  e^{-\frac{4}{p-1}s} z_1(s) \right) =
e^{-\frac{4}{p-1}s} z_2(s) \geq \delta_2 e^{-(n-2)s}
\end{equation*}
and this implies that
\begin{align*}
e^{-\frac{4}{p-1}s} z_1(s)
&\leq \frac{\delta_2}{n-2} \left( e^{-(n-2)s_0}-e^{-(n-2)s} \right)
+ e^{-\frac{4}{p-1}s_0}z_1(s_0) \\
&\leq -\delta_3 e^{-(n-2)s}
\end{align*}
for some suitable constant $\delta_3>0$.  Thus, we end up with
\begin{equation}
z_1(s) \le -\delta_3 e^{-\left( n-2-\frac{4}{p-1}\right)s} \to-\infty \mbox{\ as\ }
s\to-\infty,
\end{equation}
so that also in this case, the corresponding solution $u$ of (\ref{pde}) becomes
singular at $r=0$.
\hfill $\square$

\vspace{4mm}
\noindent
Completing the proof of Proposition~\ref{impossible} also yields the proof of
Theorem~\ref{theorem1}.

\section{The Dirichlet problem}

If we put $r=|x|$ then the equation in (\ref{Dirichlet}) becomes

\begin{equation} \label{radialeq}
u^{(4)}(r)+\frac{2(n-1)}{r}u'''(r)+\frac{(n-1)(n-3)}{r^2}u''(r)-\frac{(n-1)(n-3)}{r^3}u'(r)=\lambda(1+u)^p, \ \ \ r\in[0,1].
\end{equation}
If we put
\neweq{defU}
U(x)=1+u(x/\sqrt[4]\lambda) \qquad \qquad \mbox{for\ } x\in B_{\sqrt[4]\lambda}(0)
\endeq
then $U$ solves the equation
\begin{equation} \label{eqU}
\Delta^2 U=U^p \qquad \qquad \mbox{in} \ B_{\sqrt[4]\lambda}(0).
\end{equation}
Since the equation (\ref{eqU}) is invariant under the rescaling
$$ U_a(x)=aU(a^{\frac{p-1}{4}}x) $$
i.e. $U$ is a solution of (\ref{eqU}) if and only if $U_a$ is a solution of (\ref{eqU}),
it is not restrictive to concentrate our attention on solutions $U$ of
the equation (\ref{eqU}) which satisfy the condition $U(0)=1$.

Next we define $U_\gamma=U_\gamma(r)$ as the unique solution of the initial value problem

\neweq{initial}
\begin{tabular}{l}
$ U_\gamma^{(4)}(r)+\displaystyle{\frac{2(n-1)}{r}}U_\gamma'''(r)
+\displaystyle{\frac{(n-1)(n-3)}{r^2}}U_\gamma''(r)
\displaystyle{-\frac{(n-1)(n-3)}{r^3}}U_\gamma'(r)=|U_\gamma(r)|^{p-1}U_\gamma(r),$  \\
\\
$U_\gamma(0)=1, \qquad \qquad U'_\gamma(0)=U'''_\gamma(0)=0, \qquad \qquad U''_\gamma(0)=\gamma<0.$
\end{tabular}
\endeq

We report here the following fundamental result by \cite{GazzolaGrunau}:

\begin{lemma}[\cite{GazzolaGrunau}] \label{gazgru}
Let $n>4$ and $p>(n+4)/(n-4)$.
\begin{itemize}
\item[(i)] There exists a unique $\overline\gamma<0$ such that the solution
$U_{\overline\gamma}$ of $(\ref{initial})$ exists on the whole interval $[0,\infty)$,
it is positive everywhere, it vanishes at infinity and it satisfies
$U'_{\overline\gamma}(r)<0$ for any $r\in(0,\infty)$.
\item[(ii)] If $\gamma<\overline \gamma$ there exist $0<R_1<R_2<\infty$ such that the solution
$U_\gamma$ of $(\ref{initial})$ satisfies $U_\gamma(R_1)=0$,
$\lim_{r\uparrow R_2} U_\gamma(r)=-\infty$ and $U'_\gamma(r)<0$ for any $r\in(0,R_2)$.
\item[(iii)] If $\gamma>\overline \gamma$ there exist $0<R_1<R_2<\infty$ such that the solution
$U_\gamma$ of $(\ref{initial})$  satisfies $U'_\gamma(r)<0$ for
$r\in(0,R_1)$, $U'_\gamma(R_1)=0$, $U'_\gamma(r)>0$ for $r\in (R_1,R_2)$ and
$\lim_{r\uparrow R_2}U_\gamma(r)=+\infty$.
\item[(iv)] If $\gamma_1<\gamma_2<0$ then the corresponding solutions
$U_{\gamma_1},U_{\gamma_2}$ of $(\ref{initial})$
satisfy $U_{\gamma_1}<U_{\gamma_2}$ and $U'_{\gamma_1}<U'_{\gamma_2}$ as long as they both exist.
\end{itemize}
\end{lemma}

\begin{proof}
See the statement of \cite[Theorem 2]{GazzolaGrunau} and related proof and also the statement
of \cite[Lemma 2]{GazzolaGrunau}.
\end{proof}

For any $\gamma<0$ let $U_\gamma$ be the unique local solution of
\eq{initial}. Thanks to Lemma \ref{gazgru} (iii), for $\gamma>\overline \gamma$ we may define
$R_\gamma$ as the unique value of $r>0$ for which we have $U'_\gamma(R_\gamma)=0$.

\begin{lemma} \label{gamgam}
Let $n>4$, $p>(n+4)/(n-4)$ and $\gamma\in (\overline \gamma,0)$ with $\overline \gamma$
as in the statement
of  Lemma~\ref{gazgru}. Then the map $\gamma\mapsto R_\gamma$ is monotonically decreasing and
$$ \lim_{\gamma\downarrow \overline \gamma} R_\gamma=+\infty. $$
\end{lemma}

\begin{proof}
The fact that the map $\gamma\mapsto R_\gamma$ is monotonically decreasing follows
immediately by Lemma \ref{gazgru} (iv). This shows that the function $\gamma\mapsto R_\gamma$ admits
a limit as $\gamma\rightarrow \overline \gamma$.
Suppose by contradiction that
$$ \overline R:=\lim_{\gamma\downarrow \overline \gamma} R_\gamma<+\infty.$$
Then, by Lemma \ref{gazgru} (i), (iv) we have for all $ \gamma\in(\overline\gamma,0)$ that
\neweq{bounbelow}
U_\gamma(R_\gamma)>U_{\overline\gamma}(R_\gamma)
\geq U_{\overline\gamma}(\overline R)>0.
\endeq
Define for any $\gamma\in(\overline\gamma,0)$, $r\in[0,1]$ the function
\neweq{defugamma}
u_\gamma(r)=\frac{U_\gamma(R_\gamma r)}{U_\gamma(R_\gamma)}-1 .
\endeq
Then, $u_\gamma$ solves the Dirichlet problem
\begin{equation}\label{Dirichlet2}
\left\{
\begin{array}{ll}
\Delta^2 u_\gamma=R_\gamma^4 U_\gamma(R_\gamma)^{p-1} (1+u_\gamma)^ p \ & \mbox{in }B,\\
u_\gamma=|\nabla u_\gamma|=0 & \mbox{on }\partial B.
\end{array}
\right.
\end{equation}
Moreover, by \eq{bounbelow} and the fact that
$U_\gamma(R_\gamma)\leq U_\gamma(r)\leq U_\gamma(0)=1$
for any $r\in[0,R_\gamma]$, we have for all $\gamma\in(\overline \gamma,0)$, $x\in B$
\neweq{unifbound}
0\leq u_\gamma(x)\leq U_{\overline \gamma}(\overline R)^{-1}-1 .
\endeq
This shows that the set $\{u_\gamma:\gamma\in(\overline\gamma,0)\}$ is bounded
in $L^\infty(B)$
and hence by a bootstrap argument, from (\ref{Dirichlet2})
and the fact that $R^4_\gamma U_\gamma(R_\gamma)^{p-1}\leq \lambda^*$,
we deduce that there exists a sequence
$\gamma_k\downarrow \overline \gamma$ and a function $\overline u\in H^2_0(B)\cap C^\infty(\overline B)$
such that
\neweq{fortconv}
u_{\gamma_k}\rightarrow \overline u \qquad \qquad \mbox{in} \ C^4(\overline B)
\endeq
as $k\rightarrow \infty$.
Since the sequence $U_{\gamma_k}(R_{\gamma_k})$ is monotonically decreasing and 
bounded from below then
for any $r\in [0,\overline R)$ we have that for sufficiently large $k$,
$U_{\gamma_k}(r)=U_{\gamma_k}(R_{\gamma_k})\left[u_{\gamma_k}
\left(r/R_{\gamma_k}\right)+1\right]$
is well defined and admits a finite limit as $k\rightarrow \infty$ which will be denoted by
$\overline U(r)$.
In fact $U_{\gamma_k}\rightarrow \overline U$ in $C^4([0,R])$ for any $0<R<\overline R$
and moreover by \eq{fortconv} we also have that
$$ \overline U(x)=\left[\lim_{k\rightarrow\infty} U_{\gamma_k}(R_{\gamma_k})\right]
\cdot \left[\overline u\left(\frac{r}{\overline R}\right)+1\right].
$$
Since $\overline u \in H^2_0(B)$  we also have
\neweq{condlim}
\lim_{r\uparrow \overline R} \overline U'(r)=0.
\endeq
On the other hand by continuous dependence on the initial conditions we also have that
$$
\lim_{k\rightarrow \infty} U_{\gamma_k}(r)=U_{\overline \gamma}(r)
\qquad \qquad \mbox{for all\ } r\in[0,\overline R)
$$
and hence $\overline U(r)=U_{\overline\gamma}(r)$ for any $r\in[0,\overline R)$.
This with \eq{condlim} implies
$$ \lim_{r\uparrow \overline R}U'_{\overline\gamma}(r)=0 $$
which is absurd since $U'_{\overline\gamma}(\overline R)<0$.
This completes the proof of the lemma.
\end{proof}

\begin{lemma} \label{supcrit}
Let $n>4$ and $p>(n+4)/(n-4)$ and let
$u$ be a regular solution of $(\ref{Dirichlet})$. Then
$$
u(x)\leq \left(\frac{\lambda^*}{\lambda}\right)^{1/(p-1)}|x|^{-4/(p-1)}-1
\qquad \qquad \mbox{for all\ } x\in B\backslash\{0\}.
$$
\end{lemma}
\begin{proof}
Let $u$ be a regular solution of (\ref{Dirichlet}) for some $\lambda>0$ and define
the rescaled function
\neweq{defUbis}
U(x)=\frac{1}{1+u(0)}
\left[1+u\left(\frac{x}{\sqrt[4]\lambda(1+u(0))^{\frac{p-1}{4}}} \right)
\right]
\endeq
so that $U$ satisfies
\neweq{eqiniz}
\Delta^2 U=U^p \ \ \mbox{in} \ \ B_R(0) \qquad \mbox{and} \qquad  U(0)=1
\endeq
where we put $R=\sqrt[4]\lambda(1+u(0))^{\frac{p-1}{4}}$.

Define
$$ M=\max_{r\in[0,R]} r^{4/(p-1)}U(r)$$
and let $\overline R\in(0,R]$ be such that $\overline R^{4/(p-1)}U(\overline R)=M$.
If we define
$$ w(r)=\frac{U(\overline R r)}{U(\overline R)}-1 $$
then $w$ solves the problem

$$
\left\{
\begin{tabular}{ll}
$\Delta^2 w=\overline R^4 U(\overline R)^{p-1}(1+w)^p$ & $\qquad \mbox{in} \ B$ \\
$w=0$ & $ \qquad \mbox{on} \ \partial B$ \\
$w'\leq 0$ & $ \qquad \mbox{on} \ \partial B.$
\end{tabular}
\right.
$$
This proves that $M^{p-1}=\overline R^4 U(\overline R)^{p-1}\leq \lambda^*$
since otherwise by the super-subsolution
method (see \cite[Lemma 3.3]{bg} for more details)
we would obtain a solution of (\ref{Dirichlet}) for
$\lambda=\overline R^4 U(\overline R)^{p-1}>\lambda^*$.
This yields for all $r\in[0,R]$ that
\neweq{BOU}
U(r)\leq M r^{-4/(p-1)} \leq (\lambda^*)^{1/(p-1)} r^{-4/(p-1)} .
\endeq
Then reversing the identity \eq{defUbis}, by \eq{BOU} we obtain
$$
u(r)= \lambda^{-1/(p-1)} R^{4/(p-1)} U(Rr) -1\leq
\left(\frac{\lambda^*}{\lambda}\right)^{1/(p-1)}r^{-4/(p-1)}-1
$$
which completes the proof of the lemma.
\end{proof}

\noindent
{\it Proof of Theorem~\ref{exss}.}
For $\gamma\in(\overline \gamma,0)$ consider the corresponding solution $U_\gamma$
of the Cauchy problem
(\ref{initial}) and the function $u_\gamma$ introduced in \eq{defugamma}.
If we put $\lambda_\gamma=R_\gamma^4 U_\gamma(R_\gamma)^{p-1}$ then
by \eq{Dirichlet2} we have that $u_\gamma$ solves
\begin{equation}\label{Dirichlet3}
\left\{
\begin{array}{ll}
\Delta^2 u_\gamma=\lambda_\gamma(1+u_\gamma)^ p \ & \mbox{in }B,\\
u_\gamma=|\nabla u_\gamma|=0 & \mbox{on }\partial B.
\end{array}
\right.
\end{equation}
We show that $\lambda_\gamma$ remains bounded away from zero for
$\gamma>\overline \gamma$ sufficiently close
to $\overline \gamma$, which is defined in Lemma~\ref{gazgru}.
By \cite[Theorem 3]{GazzolaGrunau} we infer that for a fixed $\eps\in(0,K_0^{1/(p-1)})$
there exists a corresponding $r_\eps>0$ such that
\neweq{k0-eps}
U_{\overline \gamma}(r)>(K_0^{1/(p-1)}-\eps)r^{-4/(p-1)}
\qquad \qquad \mbox{for all\ } r>r_\eps.
\endeq
On the other hand, by Lemma \ref{gamgam}, we deduce that there exists $\gamma_0\in(\overline \gamma,0)$
such that for any $\gamma\in(\overline\gamma,\gamma_0)$ then $R_\gamma>r_\eps$.
Therefore by Lemma \ref{gazgru} (iv) we obtain for all $ \gamma\in(\overline\gamma,\gamma_0)$
$$ U_\gamma(R_\gamma)>U_{\overline\gamma}(R_\gamma)>(K_0^{1/(p-1)}-\eps)R_\gamma^{-4/(p-1)}
$$
and this yields
\neweq{stimlam}
 \forall \gamma\in(\overline\gamma,\gamma_0):\qquad
\lambda_\gamma>(K_0^{1/(p-1)}-\eps)^{p-1}=:C.
\endeq
Combining \eq{stimlam} and Lemma \ref{supcrit} we obtain
for all $ \gamma\in(\overline\gamma,\gamma_0)$, $x\in B\backslash\{0\}$
\neweq{unifest2}
u_\gamma(x)\leq \left(\frac{\lambda^*}{C}\right)^{1/(p-1)}|x|^{-4/(p-1)}-1.
\endeq
Since $u_\gamma$ solves \eq{Dirichlet3}, by \eq{unifest2} we obtain
$$\int_B |\Delta u_\gamma|^2 dx=\lambda_\gamma\int_B (1+u_\gamma)^p u_\gamma dx
\leq \lambda^* \int_B (1+u_\gamma)^{p+1} dx
\leq \frac{(\lambda^*)^{\frac{2p}{p-1}}}{C^{\frac{p+1}{p-1}}}
 \int_B |x|^{-\frac{4(p+1)}{p-1}} dx<+\infty
$$
since $p>(n+4)/(n-4)$.
This proves that the set $\{u_\gamma:\gamma\in(\overline\gamma,\gamma_0)\}$ is bounded
in $H^2_0(B)$ and hence there exists a sequence $\gamma_k\downarrow\overline\gamma$ and a function
$u\in H^2_0(B)$ such that $u_{\gamma_k}\weak u$ in $H^2_0(B)$.
Moreover, by (\ref{unifest2}) and applying Lebesgue's theorem,
$u$ weakly solves \eq{Dirichlet} for a suitable $\widetilde\lambda\geq C$.

It remains to prove that the function $u$ is unbounded. For simplicity, in the rest of the proof
$u_{\gamma_k}, U_{\gamma_k}, R_{\gamma_k}, \lambda_{\gamma_k}$ will be denoted respectively
by $u_k, U_k, R_k, \lambda_k$.

By compact embedding we have that $u_k\rightarrow u$ in $L^1(B)$ and hence we have
$$
\lim_{r\downarrow 0} \frac{1}{|B_r(0)|}\int_{B_r(0)} u(x) dx
=\lim_{r\downarrow 0} \left[ \frac{1}{r^n |B|}
\lim_{k\rightarrow \infty} \int_{B_r(0)} u_k(x) dx \right]
$$
and passing to radial coordinates, by \eq{defugamma} and Lemma \ref{gazgru} (iv), we obtain
\begin{eqnarray}
\lim_{r\downarrow 0} \frac{1}{|B_r(0)|}\int_{B_r(0)} u(x) dx
&=&\lim_{r\downarrow 0} \left[ -1+\frac{n}{r^n} \lim_{k\rightarrow \infty}
\int_0^r \frac{U_k(R_k\rho)}{U_k(R_k)}\rho^{n-1}d\rho \right] \notag\\
&=& \lim_{r\downarrow 0} \left[ -1+\frac{n}{r^n} \lim_{k\rightarrow \infty}
\frac{1}{R_k^nU_k(R_k)} \int_0^{R_k r} U_k(\rho)\rho^{n-1}d\rho \right]
\nonumber\\
&\geq&
\lim_{r\downarrow 0} \left[
-1+\frac{n}{r^n}
\lim_{k\rightarrow \infty}
\frac{1}{R_k^nU_k(R_k)} \int_0^{R_k r} U_{\overline\gamma}(\rho)\rho^{n-1}d\rho \right].
\label{valormedio}
\end{eqnarray}
By \eq{k0-eps} we have that there exist $C,R_0>0$ such that
\neweq{CR0}
 \forall \rho\in(R_0,\infty):\qquad
U_{\overline \gamma}(\rho)>C\rho^{-4/(p-1)}.
\endeq
Hence we have for $k>\overline k=\overline k(r)$
\neweq{stimbasso}
\int_0^{R_kr} U_{\overline \gamma}(\rho)\rho^{n-1}d\rho
\geq \int_0^{R_0} U_{\overline \gamma}(\rho)\rho^{n-1}d\rho
+C\left(n-\frac{4}{p-1}\right)^{-1}\left(R_k^{n-\frac{4}{p-1}}r^{n-\frac{4}{p-1}}-R_0^{n-\frac{4}{p-1}}
\right).
\endeq
Since $p>(n+4)/(n-4)>(n+4)/n$ and since by \eq{stimlam}, $\lambda_k$ is bounded away from zero
as $k\rightarrow\infty$ then
\begin{equation*}
\lim_{k\rightarrow\infty} R_k^nU_k(R_k)= \lim_{k\rightarrow\infty}
R_k^{n-\frac{4}{p-1}} \lambda_k^{\frac{1}{p-1}} = +\infty
\end{equation*}
and hence by \eq{stimbasso} we obtain
$$
\lim_{k\rightarrow \infty}
\frac{1}{R_k^nU_k(R_k)} \int_0^{R_k r} U_{\overline\gamma}(\rho)\rho^{n-1}d\rho
\geq
\underset{k\rightarrow\infty}{\lim\inf}
\frac{C}{\left(n-\frac{4}{p-1}\right)R_k^nU_k(R_k)}
\left(R_k^{n-\frac{4}{p-1}}r^{n-\frac{4}{p-1}}-R_0^{n-\frac{4}{p-1}} \right)
$$
\neweq{liminf}
=\underset{k\rightarrow\infty}{\lim\inf}
\frac{Cr^{n-\frac{4}{p-1}}}{\left(n-\frac{4}{p-1}\right)\lambda_k^{1/(p-1)}}
\geq
\frac{Cr^{n-4/(p-1)}}{\left(n-\frac{4}{p-1}\right)(\lambda^*)^{1/(p-1)}}
=:\widetilde C r^{n-4/(p-1)}.
\endeq
Inserting \eq{liminf} in \eq{valormedio} we obtain
$$
\lim_{r\downarrow 0} \frac{1}{|B_r(0)|}\int_{B_r(0)} u(x) dx
\geq
\lim_{r\downarrow 0}(-1+n\widetilde C r^{-4/(p-1)})=+\infty.
$$
This proves that $u\notin L^\infty(B)$.
\hfill $\square$

\bigskip


\bigskip

\noindent
{\it Proof of Theorem~\ref{theorem2}.}
We make use of an idea from \cite{DDGM}.
Let $u_{\lambda}$ denote the positive minimal regular solution
of (\ref{Dirichlet}) for $0\le\lambda<\lambda^*$. According
to \cite[Theorem 2]{FerreroGrunau}, these are stable so that one has
in particular:
$$
\forall \varphi \in C^{\infty}_0 (B):\quad  \int_B \left(\Delta\varphi(x)\right)^2\,dx
-p\lambda\int_B (1+u_\lambda(x))^{p-1}\varphi(x)^2\,dx\ge 0.
$$
By taking the monotone limit we obtain that
\begin{equation}\label{stability}
\forall \varphi \in C^{\infty}_0 (B):\quad  \int_B \left(\Delta\varphi(x)\right)^2\,dx
-p\lambda^*\int_B (1+u^*(x))^{p-1}\varphi(x)^2\,dx\ge 0.
\end{equation}
We assume now for contradiction that $u^*$ is singular. Then, according
to \cite[Theorem 5]{FerreroGrunau} we have the following
estimate from below:
\begin{equation}
u^*(x) >\left( \frac{K_0}{\lambda^*} \right)^{1/(p-1)}|x|^{-4/(p-1)}-1.
\end{equation}
Combining this with (\ref{stability}) yields
\begin{equation}
\forall \varphi \in C^{\infty}_0 (B):\quad  \int_B \left(\Delta\varphi(x)\right)^2\,dx
\ge pK_0\int_B |x|^{-4} \varphi(x)^2\,dx.
\end{equation}
However under the subcriticality assumptions made we have that
$pK_0>n^2(n-4)^2/16$. This contradicts the optimality of
the constant in Hardy's inequality
$$
\forall \varphi \in C^{\infty}_0 (B):\quad  \int_B \left(\Delta\varphi(x)\right)^2\,dx
\ge \frac{n^2(n-4)^2}{16}\int_B |x|^{-4} \varphi(x)^2\,dx,
$$
so that $u^*$ has indeed to be regular.
\hfill$\square$

\vspace{4mm}\noindent
{\bf Acknowledgement.}
We are grateful to Gianni Arioli and Filippo Gazzola for fruitful discussions
and interesting numerical experiments.

\end{document}